\documentclass[11pt]{article}
\usepackage{amsmath,amsfonts,amsthm}
\usepackage{amssymb}
\usepackage{fullpage}
\usepackage{enumerate}
\usepackage{graphicx}

\usepackage{subcaption}
\usepackage{tikz}
\usepackage{hyperref}
\hypersetup{
	colorlinks=true,
	linkcolor=blue,
}

\newtheorem{theorem}{Theorem}
\newtheorem{lemma}[theorem]{Lemma}

\newtheorem{proposition}[theorem]{Proposition}

\newtheorem{conjecture}[theorem]{Conjecture}

\renewcommand\int{\mathop{int}}

\newcommand{\etal}{\textit{et al.}}

\begin{document}

\title{Subdivision and Graph Eigenvalues}

\author{
Hitesh Kumar\thanks{Email: \small\tt hitesh\_kumar@sfu.ca},~
Bojan Mohar\thanks{Supported in part by the NSERC Discovery Grant R611450 (Canada),
and by the Research Project J1-2452 of ARRS (Slovenia). On leave from IMFM, Department of Mathematics, University of Ljubljana. Email: \small\tt mohar@sfu.ca}, ~
Shivaramakrishna Pragada\thanks{Email: \small\tt shivaramakrishna\_pragada@sfu.ca}, ~
Hanmeng Zhan\thanks{Email: \small\tt hanmeng\_zhan@sfu.ca, hanmengzhan@gmail.com}\\
\small Department of Mathematics\\[-0.8ex]
\small Simon Fraser University\\[-0.8ex]
\small Burnaby, BC \ V5A 1S6, Canada}

\date{}

\maketitle

\begin{abstract}
   This paper investigates the asymptotic nature of graph spectra when some edges of a graph are subdivided sufficiently many times. In the special case where all edges of a graph are subdivided, we find the exact limits of the $k$-th largest and $k$-th smallest eigenvalues for any fixed $k$.
   It is  expected that after subdivision, most eigenvalues of the new graph will lie in the interval $[-2,2]$. We examine the eigenvalues of the new graph outside this interval, and we prove several results that might be of independent interest. 
\end{abstract}



\section{Introduction}
Let $G$ be a finite simple graph with vertex set $V(G)$ and edge set $E(G)$. The \emph{adjacency matrix} of $G$ is defined to be the matrix $A(G)=[a_{uv}]$ where $a_{uv}=1$ if $u$ and $v$ are adjacent, and $0$ otherwise. The \emph{spectrum} of $G$ is the set of eigenvalues of $A(G)$ which 
we denote by 
\[\lambda_1(G)\geq \lambda_2(G)\geq \cdots \geq \lambda_{n-1}(G)\geq \lambda_n(G),\]
where $n=|G|.$

In the seminal paper \cite{Jiang_2021}, Jiang \etal, gave an upper bound on the multiplicity of $\lambda_2(G)$, which is denoted by $m(\lambda_2, G)$. They used it to solve a long-standing conjecture about equiangular lines. In a follow-up article \cite{Haiman_2022}, Haiman \etal, came up with a construction which showed that the upper bound on $m(\lambda_2, G)$ would be sharp if the notion of multiplicity was relaxed to the `approximate second eigenvalue multiplicity.' This construction uses Cartesian product of graphs and then a subdivision of a subset of edges to show desired spectral properties. 

Subdivision is a quintessential operation in graph theory with applications to various graph problems. Its effect on the adjacency spectrum (in particular, on the spectral radius $\lambda_1$) was studied by Hoffmann and Smith \cite{hoffmann1974spectral} who proved the celebrated `Hoffmann-Smith Subdivision Theorem.' Hoffmann initiated the study of limit points of graph spectra, known as the Hoffman program, and subdivision occurs naturally in this context. We refer the reader to the articles \cite{Guo_2008, Hoffman_1972, Venkatesan_2019, Zhang_2006}. 

 The behaviour of the spectrum studied in relation to subdivision generally falls into two broad categories: either a single edge is subdivided, or all of the edges are subdivided. However, recently, subdividing a subset of edges has found application in generating graphs with high second eigenvalue multiplicities \cite{Haiman_2022}. In this article, our goal is to analyze the asymptotic behaviour of eigenvalues of a graph $G$ when a subset $S \subseteq E(G)$ is subdivided sufficiently many times. We achieve this by comparing the eigenvalues of two graph sequences $\{G_t(S)\}_{t\in \mathbb{N}}$ and $\{H_t(S)\}_{t\in \mathbb{N}}$ which we define below.  

Let $G$ be a graph and $S \subseteq E(G).$ For $t\geq 1$, we define $G_t = G_t(S)$ to be the graph obtained from $G$ by replacing each edge $uv\in S$ with a path $P_{uv}$ of length $t$. We call $P_{uv}$ the $t$-\emph{stretch} of $uv$. In particular, $G_1=G$. We define $H_t = H_t(S)$ to be the graph obtained from $G_{2t+1}(S)$ by removing the middle edge on each path $P_{uv}$, for all $uv\in S$. Note that
\[|G_t(S)| = |G| + (t-1)|S| \quad \textrm{and} \quad |H_t(S)| = |G_{2t+1}(S)| = |G| + 2t|S|.\]

 It makes sense to compare the spectrum of graphs in the sequence $\{G_t\}_{t\in \mathbb{N}}$ and $\{H_t\}_{t\in \mathbb{N}}$. The latter sequence has the property that $H_t$ is an induced subgraph of $H_{t+1}$ for all $t$, and hence this sequence is easy to work with. Note that $G_{2t+1}$ and $H_t$ differ only by $|S|$ edges. One would expect that as $t\to \infty$, the difference in finitely many edges has a negligible effect on eigenvalues, and therefore $G_t$ and $H_t$ are `almost' co-spectral. Our main result is that for any fixed $k$, the $k$-th largest (or smallest) eigenvalues of $G_t$ and $H_t$ are asymptotically equal. More precisely, we show the following.
 
\begin{theorem}\label{thm:main1}
For every fixed $k$, $\{\lambda_k(G_t)\}_{t\in \mathbb{N}}$ and $\{\lambda_k(H_t)\}_{t\in \mathbb{N}}$ are Cauchy sequences. Furthermore, \[\lim_{t\to \infty}  \lambda_k(G_t)=\lim_{t\to \infty} \lambda_k(H_t).\]
\end{theorem}
The same result holds for the $k$-th smallest eigenvalues of these graphs. 
\begin{theorem}\label{thm:main2}
 For every fixed $k$, $\{\lambda_{|G_t|-k+1}(G_{t})\}_{t\in \mathbb{N}}$ and $\{\lambda_{|H_t|-k+1}(H_t)\}_{t\in \mathbb{N}}$ are Cauchy sequences. Furthermore, \[\lim_{t\to \infty}  \lambda_{|G_t|-k+1}(G_t)=\lim_{t\to \infty}  \lambda_{|H_t|-k+1}(H_t).\]
\end{theorem}
 
As an application of the above results, we find the limits of the $k$-th largest and $k$-th smallest eigenvalues of $G$ when all of its edges are subdivided sufficiently many times. 
\begin{theorem}\label{thm:main3}
Let $G$ be a graph on $n$ vertices with degree sequence $(d_1, d_2, \ldots, d_n)$ where $d_1\geq  d_2\geq  \cdots \geq d_n$. Let $G_t = G_t(E(G))$.  Then for any fixed $k,$
    \[\lim_{t\to \infty} \lambda_k(G_t) =-\lim_{t\to \infty} \lambda_{|G_t|-k+1}(G_t)=
    \begin{cases}
     \frac{d_k}{\sqrt{d_k-1}} \text{ if } 1\leq k \leq n \text{ and } d_k\geq 3,\\
     2 \text{ otherwise.}
    \end{cases}\]
\end{theorem}

In our proofs of the main results, we use the key observation that on `long' internal paths, the eigenvectors corresponding to `large' eigenvalues of the graph decay exponentially. In fact, we prove a more general result in \autoref{prop:eigenvector}. Moreover, we show that for eigenvalues outside $[-2,2]$, the absolute values of their eigenvector entries on the internal path are unimodal (\autoref{lemma:unimodal}). Similar results have been obtained previously for the principal eigenvector (see, for example, \cite{Tobin_2018} and \cite{Xue_2020}).

It is clear that for large $t$, most eigenvalues of $G_t$ will lie in the interval $[-2,2]$. So, the eigenvalues of interest are the ones lying outside this interval. In this article, we also explore the effect of subdivision on the number of eigenvalues outside $[-2,2]$.

The paper is organized as follows. In \autoref{prelims}, we collect some well-known theorems from the spectral theory of matrices. We discuss the eigenvector structure of `large' eigenvalues in \autoref{section:large_eigenvector}. The next section is dedicated to proving our main results, Theorems \ref{thm:main1} and \ref{thm:main2}. We prove \autoref{thm:main3} in \autoref{section:applications}. Finally, in \autoref{section:subdivision_multiplicity}, we look at the number of eigenvalues outside $[-2,2].$ and conclude the paper with some open problems.

\section{Preliminaries}\label{prelims}
We use standard graph terminology and notation throughout the paper. Here we recall some well-known results that we will use frequently. For reference, see Theorems 4.2.6, 4.3.28 and 8.4.4 of the book \cite{Horn_Johnson_2013} by Horn and Johnson.

\begin{theorem}[Courant-Fischer]
    Let $A$ be an $n\times n$ Hermitian matrix with eigenvalues $\lambda_1\ge \lambda_2\ge \cdots \ge \lambda_n$. Let $U$ denote a subspace of $\mathbf{C}^n$. Then for $1\leq k \leq n,$ 
    \[ \lambda_k = \max_{\substack{U\\ \dim(U)=k}} ~ \min_{\substack{z\in U\\||z||=1}} \langle z, A z\rangle 
    = \max_{\substack{U\\ \dim(U)=k}} \min_{z\in U} \frac{z^TAz}{z^Tz},
    \]
    and 
   \[\lambda_k =\min_{\substack{U\\ \dim(U)=n-k+1}}~ \max_{\substack{z\in U\\||z||=1}} \langle z, Az\rangle 
   = \min_{\substack{U\\ \dim(U)=n-k+1}} \max_{z\in U} \frac{z^TAz}{z^Tz}.\] 
\end{theorem}

\begin{theorem}[Interlacing]
    Let $A$ be an $n\times n$ Hermitian matrix, with eigenvalues $\lambda_1 \ge \lambda_2 \ge \cdots \ge \lambda_n$. Let $B$ be an $m\times m$ principal submatrix of $A$, with eigenvalues $\theta_1\ge \theta_2\ge \cdots \ge \theta_m$. Then
    \[\lambda_i \ge \theta_i \ge \lambda_{i+n-m}.\]
\end{theorem}

\begin{theorem}[Perron-Frobenius]
Let $A$ be an $n\times n$ ($n\geq 2$) non-negative irreducible matrix. Then there exists a real number $\rho>0$ such that the following statements hold.
\begin{enumerate}[$(i)$]
    \item $\rho$ is a simple eigenvalue of $A.$
    \item $\rho$ has a positive eigenvector (which is called Perron-vector or principal eigenvector).
    \item $\rho=\max \{|\lambda|: \lambda \text{ is an eigenvalue of $A$}\}.$
\end{enumerate}
\end{theorem}

\section{On the eigenvectors of `large' eigenvalues}\label{section:large_eigenvector}
Intuitively, the eigenvectors of a graph corresponding to its `large' eigenvalues should have small entries at vertices with small degrees. In this section, we attempt to describe the structure of eigenvectors of a graph on its subgraphs with bounded maximum degree and large diameter.

Let $G$ be a connected simple graph on $n$ vertices. We partition $V(G)$ as follows. For $U_0 \subset V(G)$ and $i\geq 1,$ define $U_i$ to be the set of vertices at distance $i$ from $U_0$. So $U_0, U_1,\ldots, U_d$ is a partition of $V(G)$ where $d$ is the maximum distance of a vertex in $G$ from $U_0$. We shall also use the symbol $U_i$ to denote the subgraph of $G$ induced by the vertices in $U_i$. Define $\partial U_0=\{u\in U_0: N(u)\cap U_1\neq \phi\}.$

\begin{proposition}\label{prop:eigenvector}
Let $G$ be a connected graph and let $U_0, U_1,\ldots, U_d$ be a partition of $V(G)$ as defined above. Suppose $\deg(v)\leq k$ for all $v\in G-U_0,$ where $k$ is some positive integer. Let  $\lambda$ be an eigenvalue of $G$ such that $|\lambda|>k$ and let $x = (x_v)_{v\in V(G)}$ be a $\lambda$-eigenvector of $G$. Let $M_0=\{|x_v|: v\in \partial U_0\}$ and for $ i\geq 1$, let $M_i=\max\{|x_v|: v\in U_i\}$. Then, for all $1\leq i \leq d,$ we have 
\begin{equation}\label{eq:smallentry1}
   M_i\leq \frac{k}{|\lambda|}M_{i-1}\leq \Bigl( \frac{k}{|\lambda|}\Bigr)^i M_0.
\end{equation}
\end{proposition}
\begin{proof} Suppose $v\in U_i$ is such that $|x_v|=M_i$. We can assume that $M_i>0$, or else we are done. By eigenvalue-eigenvector equation,
\begin{align}\label{eq:smallentry2}
   |\lambda| M_i =|\lambda x_v| & = \bigg|\sum_{\substack{u\sim v\\u\in U_{i-1}}}x_u + \sum_{\substack{u\sim v\\u\in U_{i}}}x_u + \sum_{\substack{u\sim v\\u\in U_{i+1}}}x_u \bigg| \nonumber \\
    & \leq  \sum_{\substack{u\sim v\\u\in U_{i-1}}}|x_u| + \sum_{\substack{u\sim v\\u\in U_{i}}}|x_u| + \sum_{\substack{u\sim v\\u\in U_{i+1}}}|x_u|\nonumber \\ 
   & \leq \deg_{U_{i-1}}(v)M_{i-1}+\deg_{U_i}(v)M_i +\deg_{U_{i+1}}(v)M_{i+1}.
\end{align}

 We prove the claim by induction on $i$. The base case is $i=d$. If $M_d>M_{d-1}$ then \eqref{eq:smallentry2} gives $|\lambda| M_d < k M_d$ which contradicts the assumption that $|\lambda|>k$. (Note that $U_{d+1}$ is an empty set.) Therefore, $M_d\leq M_{d-1}$ and by \eqref{eq:smallentry2} we have 
\[
    |\lambda| M_d  \leq  (\deg_{U_{d-1}}(v)+\deg_{U_d}(v))M_{d-1} \leq  kM_{d-1}.
\]
For the induction step, let the claim be true for $i+1, \ldots, d$. By \eqref{eq:smallentry2} and induction hypothesis, 
\begin{equation}\label{eq:smallentry3}
    |\lambda| M_i \leq \deg_{U_{i-1}}(v)M_{i-1}+(\deg_{U_i}(v) +\deg_{U_{i+1}}(v))M_{i}.
\end{equation}
It is easy to see that if $M_i> M_{i-1}$, we must have $|\lambda|<k,$ a contradiction. So $M_i\leq M_{i-1}$. Then \eqref{eq:smallentry3} implies $|\lambda| M_i \leq (\deg_{U_{i-1}}(v)+\deg_{U_i}(v) +\deg_{U_{i+1}}(v))M_{i-1} \leq k M_{i-1}$. 
\end{proof}

We require the following special case of the above proposition to prove our main results. We say a path $P=v_0v_1\ldots v_s$ in a graph $G$ is an \emph{internal path} if $\deg(v_i)=2$ for $1\leq i \leq s-1.$ 

\begin{lemma}[Exponential Decay Lemma]\label{lemma:decay}
Let $P=v_0v_1\ldots v_s$ be an internal path in a graph $G$. Suppose $\lambda$ is an eigenvalue of $G$ such that $|\lambda|>2$ and let $x = (x_v)_{v\in V(G)}$ be a $\lambda$-eigenvector with $||x||=1$. Let $x_j = x_{v_j}$ for $0\le j\le s$. Define $M_j=\max\{|x_j|, |x_{s-j}|\}$. Then, for $1\leq j \leq \lfloor \frac{s}{2}\rfloor$,
     \begin{equation}\label{eq:smallentry4}
     M_j\leq \bigg(\frac{2}{|\lambda|}\bigg)^j.    
     \end{equation}
Moreover, for $\varepsilon > 0,$     
\begin{enumerate}[$(a)$]
      \item
     if $\lfloor \frac{s}{2}\rfloor \ge p\ge\frac{\log\big(\frac{2|\lambda|}{\varepsilon(|\lambda|-2)}\big)}{\log(|\lambda|/2)}$, then 
     \[\sum_{j=p}^{s-p} |x_j| < \varepsilon.\]
     \item
     if $\lfloor \frac{s}{2}\rfloor \ge p\ge\frac{\log\big(\frac{2|\lambda|^2}{\varepsilon(|\lambda|^2-4)}\big)}{2\log(|\lambda|/2)}$, then 
     \[\sum_{j=p}^{s-p} |x_j|^2 < \varepsilon.\]
\end{enumerate}
\end{lemma}
\begin{proof} Note that $V(G)$ can be partitioned into sets $U_0,\ldots, U_{\lfloor\frac{s}{2}\rfloor}$ where $U_i=\{v_i, v_{s-i}\}$ and $U_0=V(G)- \cup_i U_i$. Then by \autoref{prop:eigenvector}, inequality \eqref{eq:smallentry4} follows immediately. For part $(a)$ observe that 
\[
\sum_{j=p}^{s-p} |x_j| \leq 2\sum_{j=p}^{\lfloor\frac{s}{2}\rfloor} M_j \leq 2\sum_{j=p}^{\lfloor\frac{s}{2}\rfloor} \big(2/|\lambda|\big)^j < 2\frac{(2/|\lambda|)^p}{1-2/|\lambda|}
\]
which is less than $\varepsilon$ when $p\ge \frac{\log\big(\varepsilon(1/2-1/|\lambda|)\big)}{\log(2/|\lambda|)}=\frac{\log\big(\frac{2|\lambda|}{\varepsilon(|\lambda|-2)}\big)}{\log(|\lambda|/2)}$.

Similarly, for part $(b),$
\[\sum_{j=p}^{s-p} |x_j|^2  <  2\frac{(2/|\lambda|)^{2p}}{1-(2/|\lambda|)^2} \]
which is less than $\varepsilon$ when $p\ge \frac{\log\big(\varepsilon(1/2-2/|\lambda|^2)\big)}{2\log(2/|\lambda|)}=\frac{\log\big(\frac{2|\lambda|^2}{\varepsilon(|\lambda|^2-4)}\big)}{2\log(|\lambda|/2)}$.
\end{proof}

In fact, for internal paths, we can prove the following unimodality result about eigenvectors corresponding to eigenvalues with absolute value bigger than 2.

\begin{lemma}[Unimodality Lemma]\label{lemma:unimodal}
 Let $P=v_0v_1\ldots v_s$ be an internal path in a graph $G$. Suppose $\lambda$ is an eigenvalue of $G$ such that $|\lambda|\ge 2$ and let $\mu=|\lambda|-1.$ Let $x$ be a $\lambda$-eigenvector of $G$ and let $x_j$ be the entry of $x$ corresponding to $v_j$ for $0\le j\le s$. Suppose $k$ is such that 
 \[|x_k|=\min\{|x_j|: 0\le j\le s\}.\] 
Then for $0\le i<j\le k-1$ we have
     \[|x_i| \ge \mu^{j-i} |x_j|,\]
and for $k+1\le i< j\le s$ we have
     \[|x_j| \ge \mu^{j-i} |x_i|.\]
In other words, the vector $|x|$ is unimodal over the internal path. 
\end{lemma}
\begin{proof} Recall the eigenvalue equation $\lambda x_j = x_{j-1} + x_{j+1}$ which holds for each $1\leq j\leq s-1$. By the triangle inequality,
\begin{equation}\label{eq:convex}
2|x_j| \le |\lambda| |x_j| \le |x_{j-1}| + |x_{j+1}|,    
\end{equation}
so $|x|$ is a `convex function' on the internal path, and it has a minimum, say $|x_k|$. For $0\le i< j\le k-1$, inequality \eqref{eq:convex} gives the following:
\[
|x_i| \ge |\lambda| |x_{i+1}| - |x_{i+2}| = \mu |x_{i+1}|+(|x_{i+1}|-|x_{i+2}|)\ge \mu |x_{i+1}|.\]
By repeating the same for $i+1, \ldots, j-1,$ we conclude that $|x_i|\ge \mu^{j-i}|x_j|$.
The proof for $k+1 \le i< j \le s$ is the same.
\end{proof}

It is interesting to compare the above lemma with the following known result about the unimodality of the principal eigenvector on an internal path of a graph. The authors in \cite{Xue_2020} used it to minimize spectral radius over unicyclic graphs with given girth.

\begin{lemma}\cite[Lemma 1]{Xue_2020} \label{lemma:principalunimodal}
Let $P = v_0v_1\dots v_s$ be an internal path in $G$. Let $x$ be the principal eigenvector of $G$, and suppose $x_i$ is the entry of $x$ corresponding to $v_i$. If $\lambda_1(G)>2$, then the following statements hold. 
\begin{enumerate}[$(a)$]
    \item If $x_0 = x_s$, then $x_0 > x_1 > \dots > x_{\lfloor{s/2}\rfloor} = x_{\lceil{s/2}\rceil} <\dots < x_{s-1} < x_{s}$ and $x_i=x_{s-i}$ for $0 \leq i \leq s$.
    \item Suppose $x_0 < x_s$. If $x_0 \leq x_1$, then $x_0 \leq x_1 < x_2< \dots < x_{s-1} < x_s$. If $x_0>x_1$, then $x_{s-i}>x_i$ for $0\leq i \leq \lceil{s/2}\rceil -1$, and there exists an integer $1\leq k \leq \lfloor{s/2}\rfloor$ such that  $x_0 > x_1 > \dots > x_{k-1} \geq x_{k} <\dots < x_{s-1} < x_{s}$.
\end{enumerate}
\end{lemma}

In a sense, \autoref{lemma:unimodal} improves upon the above \autoref{lemma:principalunimodal} in two ways. Firstly, it extends the unimodality result to all eigenvalues of the graph outside $(-2,2)$. Secondly, it estimates how consecutive entries compare to each other in terms of the factor $\mu=|\lambda|-1$. The proof of \autoref{lemma:principalunimodal} given in \cite{Xue_2020} uses the solution of the recurrence relation of the path. We offer another proof which may be of independent interest. It is based on \autoref{lemma:unimodal} and Courant-Fischer Theorem. 

\begin{proof}[Proof of \autoref{lemma:principalunimodal}]
Let us assume that $||x||=1$. First, consider the case $x_0=x_s$. Suppose to the contrary, there exists $j$ such that
$j=\min\{i : x_i\neq x_{s-i}, 0\le i \le \lceil \frac{s}{2}\rceil -1 \}$. Clearly, $j>0$. Define a new graph $G'=G-v_{j-1}v_j-v_{s-j+1}v_{s-j}+v_{j-1}v_{s-j}+v_{s-j+1}v_j$. Clearly $G'\cong G$. By Courant-Fischer Theorem,
\begin{align*}
  \lambda_1(G') &\geq x^TA(G')x\\
  &=x^TA(G)x-x_{j-1}x_j-x_{s-j+1}x_{s-j}+x_{j-1}x_{s-j}+x_{s-j+1}x_j\\ 
  &=\lambda_1(G)=\lambda_1(G')
\end{align*}
since $x_{j-1}=x_{s-j+1}$. It follows that $x$ is an eigenvector of $G'$. This contradicts the fact that the principal eigenvector of $G$ is unique. Hence, no such $j$ exists. Then part $(a)$ follows by \autoref{lemma:unimodal}.

To prove $(b)$, suppose that $x_0<x_s$. If $x_0\leq x_1$ then by \autoref{lemma:unimodal}, $x_0=\min\{x_i: 0\le i\le s\}$ and $x_0 \leq x_1 < x_2< \dots < x_{s-1} < x_s$. On the other hand, if $x_0>x_1$ then let $j=\min\{i : x_i\geq  x_{s-i}, 0\le i \le \lceil \frac{s}{2}\rceil-1\}$. Clearly, $j>0$. As before, consider the graph $G'=G-v_{j-1}v_j-v_{s-j+1}v_{s-j}+v_{j-1}v_{s-j}+v_{s-j+1}v_j$. Then the same proof as above shows that $x$ is an eigenvector of $G'$ which contradicts the uniqueness of the principal eigenvector of $G$. Therefore, no such $j$ exists. Now if $k$ is such that $x_k=\min\{x_i: 0\le i\le s\}$, then $1\le k \le \lfloor \frac{s}{2}\rfloor$. Part $(b)$ is then immediate from \autoref{lemma:unimodal}.
\end{proof}

\section{Eigenvalues of $G_t$ and $H_t$} \label{section:GtHt}
This section is devoted to the proof of Theorems \ref{thm:main1} and \ref{thm:main2}. Let us recall that the graphs $G_t$ and $H_t$ are defined for every $t\ge 1$ by subdividing a prefixed set $S\subset E(G)$.

Observe that $H_t$ is an induced subgraph of $G_{2t+2}$ for all $t$. Hence, $\lambda_k(G_{2t+2})\ge \lambda_k(H_t)$ by Interlacing Theorem. In the following lemmas we show that $\lambda_k(H_t)$ cannot be much smaller than $\lambda_k(G_{2t+2})$ for large $t$.

\begin{lemma}\label{lem:c1} Let $k\in \mathbb{N}$ and $\eta \in \mathbb{R}^+$. Suppose $\{t_\ell\}_{\ell\in \mathbb{N}}$ is a sequence of natural numbers such that 
\[\lambda_k(G_{2t_\ell+2})\ge 2+\eta\] 
for all $\ell$. Then for every $\varepsilon>0$ there exists $\ell_0$ such that for all $\ell\geq \ell_0$ we have
\[ \lambda_k(H_{t_\ell})\ge 2+\eta -\varepsilon.\]
\end{lemma}
\begin{proof} Let $x^{(i)}=(x^{(i)}_v)_{v\in V(G_{2t+2})}$ be a normalized eigenvector for $\lambda_i(G_{2t+2})$. For an edge $uv$ in $S$, let $a,b,c$ be the middle vertices in the $t$-stretch of $uv$. Define $H_t'$ as the disjoint union of $H_t$ and $|S|$ many copies of $K_1$. Then for $1\leq i\leq j\leq k$ we have
\begin{align*}
     \langle x^{(i)}, A(H_t') x^{(j)} \rangle &= \langle x^{(i)}, A(G_{2t+2}) x^{(j)} \rangle - \sum_{uv\in S} \Bigl( x^{(i)}_ax^{(j)}_b + x^{(i)}_bx^{(j)}_a +  x^{(i)}_cx^{(j)}_b + x^{(i)}_bx^{(j)}_c \Bigr)\\
     &\ge \delta_{ij} \lambda_i(G_{2t+2}) - \sum_{uv\in S} \Bigl( |x^{(i)}_a||x^{(j)}_b| + |x^{(i)}_b||x^{(j)}_a|+|x^{(i)}_c||x^{(j)}_b| + |x^{(i)}_b||x^{(j)}_c|\Bigr)\\
     &\ge \delta_{ij} \lambda_i(G_{2t+2}) - \sum_{uv\in S}\Bigl(|x^{(i)}_a|+|x^{(i)}_b|+|x^{(i)}_c|\Bigr) \Bigl(|x^{(j)}_a| + |x^{(j)}_b|+|x^{(j)}_c|\Bigr).
\end{align*}

We may assume that $\varepsilon\le (k^2|S|)^{-1}$. By \autoref{lemma:decay}$(a)$, $|x^{(i)}_a|+|x^{(i)}_b|+|x^{(i)}_c|<\varepsilon$ for all large $t$ in our sequence, since $\lambda_i(G_{2t_\ell+2})\ge 2+\eta$, for all $i\leq k$. Thus we have
\[ \langle x^{(i)}, A(H_t') x^{(j)} \rangle    \ge \delta_{ij} \lambda_i(G_{2t+2}) - |S| \varepsilon^2.\]

Now let $W$ be the subspace spanned by $x^{(1)},\ldots,x^{(k)}$. Let 
\[z=\sum_{i=1}^k c_i x^{(i)},\quad \sum_{i=1}^k c_i^2=1.\]
Then,
\begin{align*}
    \langle z, A(H_t')z \rangle &= \sum_{i=1}^k \sum_{j=1}^k c_i c_j \langle x^{(i)}, A(H_t') x^{(j)} \rangle \\
    &\ge \sum_{i=1}^k c_i^2 \lambda_i(G_{2t+2}) - \sum_{i=1}^k \sum_{j=1}^k |c_i||c_j ||S| \varepsilon^2\\
    &\ge \lambda_k(G_{2t+2}) - \left(\sum_{i=1}^k |c_i|\right)^2 |S| \varepsilon^2\\
    &\ge \lambda_k(G_{2t+2}) - k^2 |S| \varepsilon^2.
\end{align*}
Thus by Courant-Fischer Theorem,
\begin{align*}
    \lambda_k(H_t) &= \lambda_k(H_t') \\
    &= \max_{\substack{U\\ \dim(U)=k}} \min_{\substack{z\in U\\||z||=1}} \langle z, A(H_t') z\rangle\\
    &\ge \min_{\substack{z\in W\\ ||z||=1}}\langle z, A(H_t') z\rangle\\
    &\ge \lambda_k(G_{2t+2}) - k^2|S| \varepsilon^2 \ge 2+\eta-\varepsilon.
\end{align*}
This completes the proof.
\end{proof}

The following is also immediate from the proof of \autoref{lem:c1}.

\begin{lemma}\label{lem:c2} Let $k\in \mathbb{N}$. Suppose there exist positive constants $t_0$ and $\eta$ such that 
\[ \lambda_k(G_{2t+2})\geq 2+\eta\]
for all $t\geq t_0$. Then for every $\varepsilon > 0$ there exists $t_1$ such that whenever $t\geq t_1,$ 
\[ \lambda_k(H_t)\ge \lambda_k(G_{2t+2})-\varepsilon. \]
\end{lemma}

We now prove our first main result.

\begin{proof}[Proof of \autoref{thm:main1}]
We first show that 
\begin{equation}\label{eq:1evensequence}
  \lim_{t\to \infty}\lambda_k(G_{2t+2})=\lim_{t\to \infty}\lambda_k(H_t).  
\end{equation}
Note that $\{H_t\}_{t\in \mathbb{N}}$ is a sequence of graphs with bounded maximum degree and that $H_t$ is an induced subgraph of $H_{t+1}$. Therefore, $\lambda_k(H_t)$ is an increasing bounded sequence, hence Cauchy. Since the path $P_t$ is an induced subgraph of $H_t$, we have that $\lambda_k(H_t)\ge \lambda_k(P_t)$ for all $t$ by Interlacing Theorem. Thus, 
\[\lim_{t\to \infty}\lambda_k(H_t)\geq \lim_{t\to \infty}\lambda_k(P_t)=2.\]

First suppose that $\lim_{t\to \infty}\lambda_k(H_{t})=2$. By \autoref{lem:c1} there exists a non-negative decreasing function $f(t)$ such that for all large $t$ we have $\lambda_k(G_{2t+2})\leq 2+f(t)$ and $f(t)\to 0$ as $t\to \infty$.  But $\lambda_k(G_{2t+2})\geq \lambda_k(H_t)$ and thus by Sandwich Theorem for limits we get $\lim_{t\to \infty}\lambda_k(G_{2t+2})=2$.

Now suppose that $\lim_{t\to \infty}\lambda_k(H_{t})=2+2\eta$ where $\eta$ is some positive constant. Then there exists $t_0$ such that for all $t\geq t_0$ we have $\lambda_k(G_{2t+2})\geq \lambda_k(H_{t})> 2+\eta$. Now for any $\varepsilon>0$, by \autoref{lem:c2}, we can find $t_1$ such that for all $t\geq t_1,$
\[ |\lambda_k(G_{2t+2})-\lambda_k(H_{t})|<\varepsilon.\]
Thus $ \lim_{t\to \infty}\lambda_k(G_{2t+2})=2+2\eta$. This completes the proof of equality \eqref{eq:1evensequence}.

Now it remains to show that the limit over odd values $2t+3$ is the same:
\begin{equation}\label{eq:1oddsequence}
 \lim_{t\to \infty}\lambda_k(G_{2t+3})=\lim_{t\to \infty}\lambda_k(H_t).
\end{equation}
 Let $G'$ be the graph obtained from $G$ by subdividing each edge $uv\in S$ once; let $w$ be the vertex of subdivision. Let $S'$ be the subset of edges of $G'$ that consists of $uw$ for each $uv\in S$. Let $H_t' = H_t(S')$ be the graph constructed from $G'$. Our discussion above shows that
 \begin{equation}\label{eq:1odd1}
  \lim_{t\to \infty}\lambda_k(G'_{2t+2})=\lim_{t\to \infty}\lambda_k(H_t').   
 \end{equation}
Now observe that $H_t$ is an induced subgraph of $H_t'$, and $H_t'$ is an induced subgraph of $H_{t+1}$. So by Interlacing Theorem, 
\begin{equation}\label{eq:1odd2}
    \lambda_k(H_t)\le \lambda_k(H_t')\le \lambda_k(H_{t+1}).
\end{equation}
Equality \eqref{eq:1oddsequence} now follows from \eqref{eq:1odd1}, \eqref{eq:1odd2} and the fact that $G_{2t+2}'=G_{2t+3}$. This completes the proof of the theorem.
\end{proof}

\autoref{thm:main2} can be proved analogously. We give the details in the Appendix.

\section{Applications}\label{section:applications}
In this section, we use our main results to determine the exact eigenvalue limits in the special case when all edges of a graph are subdivided. 

Consider a graph $G$, and subdivide each of its edges $t$ times. This generates a sequence $\{G_t=G_t(E(G))\}_{t\in \mathbb{N}}$  of graphs, which may or may not converge in the graph sense \cite{Mohar_1982}. But by our main results, we know that $k$-th largest and $k$-th smallest eigenvalues of $G_t$ do converge. 

We first consider the star graph $K_{1,d}$ on $d+1$ vertices. We start by finding the limit of the spectral radius of $(K_{1,d})_t$; these graphs are called \emph{spiders} or \emph{star-like trees}, and their spectral radii are shown to be less than $d/\sqrt{d-1}$, see \cite{Oboudi_2018}. We prove that this bound is indeed the limit of $\lambda_1((K_{1,d})_t)$ as $t \to \infty.$ To this end, we need a useful relation between characteristic polynomials and walk-generating functions.

For a graph $G$ with adjacency matrix $A$, the characteristic polynomial of $G$ is given by
\[\phi(G,x) = \det(xI - A).\]
For any vertex $a$ in $G$, the number of closed walks of length $\ell$ and starting at $a$ is given by $A^{\ell}_{aa}$. It is helpful to consider the walk-generating function
\[W_{aa}(G, x) = I_{aa}+xA_{aa} + x^2A^2_{aa} + \cdots = (I-xA)^{-1}_{aa}.\]

\begin{theorem}[\cite{Godsil_1993}, Ch.5]\label{thm:ratfun}
Let $G$ be a graph. Let $W_{aa}(G, x)$ be the generating function of closed walks at vertex $a$ in $G$. Let $\phi(G, x)$ be the characteristic polynomial of $G$. Then
\[\frac{\phi(G\backslash a, x)}{\phi(G, x)}=x^{-1} W_{aa}(G, x^{-1}).\]
Moreover, $\phi(G\backslash a, x)/\phi(G, x)$ is decreasing in every interval where it is defined.
\end{theorem}

We apply the above theorem to the paths.
\begin{lemma}\label{lem:ratlim}
Let $\phi(P_t, x)$ be the characteristic polynomial of $P_t$. For $x\in(2,\infty)$, we have
\[\lim_{t\to \infty}\frac{\phi(P_{t-1}, x)}{\phi(P_t, x)} = \frac{1}{2}(x-\sqrt{x^2-4}).\]
\end{lemma}
\begin{proof}
Let $W_{11}(P_t, x)$ be the generating function of closed walks at the first vertex of $P_t$. By \autoref{thm:ratfun}, 
\[\frac{\phi(P_{t-1}, x)}{\phi(P_t, x)} = x^{-1} W_{11}(P_t, x^{-1}),\]
which is a formal power series. We show that it converges as $t \to \infty$. Indeed, for any positive integers $m$ and $n$ with $m\ge n$, the paths $P_m$ and $P_n$ have the same number of closed walks of length $\ell$ at the first vertex, provided that $\ell \le 2n-1$. Hence, $W_{11}(P_m, x)$ and $W_{11}(P_n, x)$ agree at the first $2n$ terms. It follows that $\{\phi(P_{t-1}, x)/\phi(P_t, x)\}_{t\ge 2}$ is Cauchy and converges to a limit as a formal power series.

Next, note that $\phi(P_t, x)$ satisfies a three-term recurrence:
\[\phi(P_{t+1}, x)=x\phi(P_t, x) - \phi(P_{t-1}, x),\]
which implies 
\[\frac{\phi(P_{t+1}, x)}{\phi(P_t, x)} = x - \frac{\phi(P_{t-1},x)}{\phi(P_t, x)}.\]
Thus when $x>2$, we can solve for the limit of $\phi(P_{t-1}, x)/\phi(P_t, x)$; in particular, since the ratio $\phi(P_{t-1}, x)/\phi(P_t, x)$ is decreasing in $(2,\infty)$, the limit must be $\dfrac{1}{2}(x-\sqrt{x^2-4})$.
\end{proof}

Now we are ready to find the limit of the spectral radii of the spiders $(K_{1,d})_t$.
\begin{lemma}\label{lemma:spider} For every $d\ge 2$, we have
\[\lim_{t\to\infty} \lambda_1 ((K_{1,d})_t) = \frac{d}{\sqrt{d-1}}.\]
\end{lemma}
\begin{proof}
When $d=2$, $(K_{1,2})_t=P_{2t+1}$ and so $\lim_{t\to\infty} \lambda_1 ((K_{1,d})_t) = \lim_{t\to\infty} \lambda_1 (P_{2t+1}) =2.$

Now suppose $d>2$. Let $P_{d, t+1}$ be a weighted path on vertices $\{0,1,\cdots, t\}$, where the first edge $(0,1)$ has weight $\sqrt{d}$, and all other edges have weight $1$. Then $P_{d, t+1}$ is a quotient graph of $(K_{1,d})_t$, so they have the same spectral radius. Consider the characteristic polynomial $\phi(P_{d, t+1}, x)$ of $P_{d,t+1}$. Expanding $\det(xI-A(P_{d,t+1}))$ gives
\[\phi(P_{d,t+1}, x) = x\phi(P_t, x) - d\phi(P_{t-1}, x).\]
Thus, $\lambda$ is an eigenvalue of $P_{d,t}$ if and only if it is a root of 
\[\frac{x}{d} = \frac{\phi(P_{t-1}, x)}{\phi(P_t, x)}.\]
Note that the right-hand side is a rational function that strictly decreases in every interval where it is defined and has exactly one branch in $(2, \infty)$. By \autoref{lem:ratlim}, when $x>2$, the right hand side converges to the value
\[\frac{1}{2}(x-\sqrt{x^2-4}).\]
Thus the limit of the spectral radius of $P_{d,t}$, as $t \to \infty$, is the positive root of 
\[\frac{x}{d} = \frac{1}{2} (x-\sqrt{x^2-4})\]
that is, $d/\sqrt{d-1}$.
\end{proof}

We now know the limits of (almost) all eigenvalues of $(K_{1,d})_t$ as $t\to \infty$. We summarize them below.
\begin{proposition}\label{prop:spiderlimits}
We have 
\begin{equation}\label{eq:spider1}
    \lim_{t\to\infty} \lambda_1 ((K_{1,d})_t) =-\lim_{t\to\infty} \lambda_{|(K_{1,d})_t|} ((K_{1,d})_t)=
\begin{cases}
    2 \text{ if }d=1,2,\\
    \frac{d}{\sqrt{d-1}} \text{ if } d>2.
\end{cases}
\end{equation}
And for any fixed $k\ge 2$,
\begin{equation}\label{eq:spider2}
    \lim_{t\to\infty} \lambda_k ((K_{1,d})_t)=-\lim_{t\to\infty} \lambda_{|(K_{1,d})_t|-k+1} ((K_{1,d})_t)=2. 
\end{equation}
\end{proposition}
\begin{proof} Note that spider graphs are bipartite and hence its $k$-th largest and $k$-th smallest eigenvalues are equal. Then equality \eqref{eq:spider1} is immediate from \autoref{lemma:spider}. Now if $k\ge 2$ then \[\lambda_k(P_t)\le \lambda_k((K_{1,d})_t)\le \lambda_2((K_{1,d})_t)\le \lambda_1(P_t).\]
The last inequality follows by Interlacing Theorem since deleting the central vertex of the spider gives a collection of disjoint paths of length $t$. Then equality \eqref{eq:spider2} follows by Sandwich Theorem for limits.   
\end{proof}

Let $G$ be a graph with degree sequence $(d_1, d_2, \ldots, d_n)$ where $d_1\geq  d_2\geq  \cdots \geq d_n$. Let $G_t = G_t(E(G))$. Then $H_t$ is a disjoint union of the spiders $(K_{1,d_1})_t, (K_{1,d_2})_t,\ldots,(K_{1,d_n})_t.$ Then the above \autoref{prop:spiderlimits}, together with Theorems \ref{thm:main1} and \ref{thm:main2}, gives us \autoref{thm:main3}.

\section{Subdivision and eigenvalue interval multiplicity} 
\label{section:subdivision_multiplicity}

For a graph $G$, define $G_t$ and $H_t$ as in previous sections for a fixed subset $S\subseteq E(G)$. It is clear that the path $P_t$ is an induced subgraph of both $G_t$ and $H_t$; thus, by Interlacing Theorem, most eigenvalues of $G_t$ and $H_t$ lie inside the interval $[-2,2]$. So the eigenvalues of interest are the ones outside this interval.

In \cite{Haiman_2022}, the authors relaxed the notion of eigenvalue multiplicity to counting the number of eigenvalues in subsets of $\mathbb{R}$; this gives us a formulation to talk about eigenvalue multiplicities with regards to subdivision. Given $I \subset \mathbb{R}$, let $m_G(I)$ be the number of eigenvalues of $G$ (counting with multiplicities) that lie in $I$. In particular, 
$$ m_G(a,b) := \# \{i : \lambda_i(G) \in (a, b) \}. $$

As we are interested in eigenvalues outside the interval $[-2,2]$, we also have:
\begin{align*}
  m_G(2,\infty) & =\#\{i : \lambda_i(G) > 2\},\\
  m_G(-\infty,-2) & =\#\{i : \lambda_i(G) < -2\}.
\end{align*}

We first show that subdivision cannot decrease the number of eigenvalues of a graph inside the interval $(2,\infty)$.
\begin{proposition}\label{prop:sub_M} 
Let $G'$ be a graph obtained from the graph $G$ by subdividing an edge. If $\lambda_k(G)>2$, then $\lambda_k(G')>2.$ In other words,
     \[m_G(2,\infty)\leq m_{G'}(2,\infty).\]
\end{proposition}
\begin{proof} Suppose $\lambda_k(G)>2$. Let the vertex of subdivision be $w$ in $G'$. Let $x^{(i)}$ be the normalized eigenvectors corresponding to $\lambda_i(G)$ for $1\leq i\leq k$. Define vectors $z^{(i)}$ such that 
\[z^{(i)}_a=\begin{cases} x^{(i)}_a \text{ if }a\neq w,\\
x^{(i)}_u \text{ if } a=w,
\end{cases}\]
where $u$ is one of the two neighbours of $w$. Let $U$ be the span of the vectors $z^{(i)}$ $(1\le i \le k)$ and let $z=\sum_i c_iz^{(i)}$ be a vector in $U$ such that $||z||^2=1+z_w^2$.  Take $x=\sum_i c_ix^{(i)}$. Then, 
\begin{align*}
   \frac{z^TA(G')z}{z^Tz}
   & = \frac{x^TA(G)x + 2z_w^2}{1 + z_w^2}\\
   & \geq \frac{\lambda_k(G)+ 2z_w^2 }{1 + z_w^2}\\
   & > \frac{2+ 2z_w^2}{1 + z_w^2} = 2.
\end{align*}
Thus by Courant-Fischer Theorem,
\[ \lambda_k(G')\ge \min_{z\in \substack{U}} \frac{z^TA(G')z}{z^Tz} > 2. \]
This completes the proof.
\end{proof}

The above proposition does not hold for $m_G(-\infty,-2)$. For example, let $G$ be the graph $C_4$ with a pendant edge. The following figure illustrates that  $m_G(-\infty,-2)$ can increase or decrease when $G$ is subdivided.

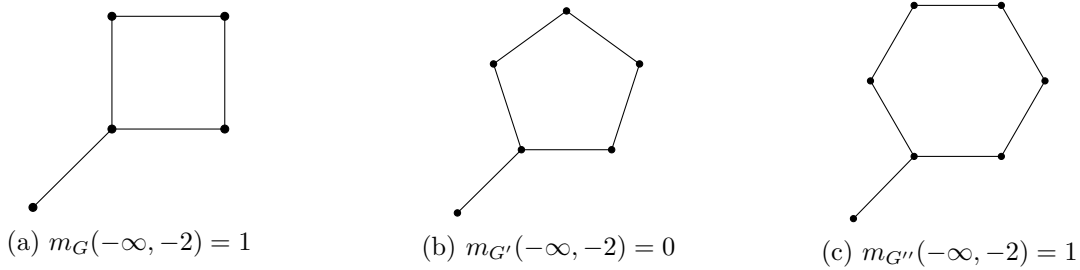
\begin{figure}[h!]
\begin{subfigure}{0.33\textwidth}
\centering
\begin{tikzpicture}[scale=0.75]
\draw  (0,2)-- (2,2);
\draw  (2,2)-- (2,0);
\draw  (2,0)-- (0,0);
\draw  (0,0)-- (0,2);
\draw  (0,0)-- (-1.4,-1.39);
\draw[fill=black] (0,2) circle (2pt);
\draw[fill=black] (2,2) circle (2pt);
\draw[fill=black]  (2,0) circle (2pt);
\draw[fill=black]  (0,0) circle (2pt);
\draw[fill=black] (-1.4,-1.39) circle (2pt);
\end{tikzpicture}
\caption{$m_{G}(-\infty,-2)=1$}
\end{subfigure}
\begin{subfigure}{0.33\textwidth}
\centering
\begin{tikzpicture}[scale=0.6]
\draw  (8,0)-- (10,0); 
\draw  (10,0)-- (10.618033988749895,1.9021130325903064);
\draw  (10.618033988749895,1.9021130325903064)-- (9,3.0776835371752522);
\draw  (9,3.0776835371752522)-- (7.381966011250105,1.9021130325903073);
\draw  (7.381966011250105,1.9021130325903073)-- (8,0);
\draw  (8,0)-- (6.582506095870898,-1.4015172463811394);   
\draw[fill=black]  (8,0) circle (2pt);
\draw[fill=black]  (10,0) circle (2pt);
\draw[fill=black]  (10.618033988749895,1.9021130325903064) circle (2pt);
\draw[fill=black]  (9,3.0776835371752522) circle (2pt);
\draw[fill=black]  (7.381966011250105,1.9021130325903073) circle (2pt);
\draw[fill=black] (6.582506095870898,-1.4015172463811394) circle (2pt);
\end{tikzpicture}
\caption{$m_{G'}(-\infty,-2)=0$}
\end{subfigure}
\begin{subfigure}{0.3\textwidth}
\centering
\begin{tikzpicture}[scale=0.58]
\draw  (16,0)-- (18,0);
\draw  (18,0)-- (19,1.7320508075688774);
\draw  (19,1.7320508075688774)-- (18,3.4641016151377553);
\draw  (18,3.4641016151377553)-- (16,3.4641016151377557);
\draw  (16,3.4641016151377557)-- (15,1.732050807568879);
\draw  (15,1.732050807568879)-- (16,0);
\draw  (16,0)-- (14.60929333999749,-1.427493904129122);
\draw[fill=black]  (16,0) circle (2pt);
\draw[fill=black]  (18,0) circle (2pt);
\draw[fill=black]  (19,1.7320508075688774) circle (2pt);
\draw[fill=black]  (18,3.4641016151377553) circle (2pt);
\draw[fill=black]  (16,3.4641016151377557) circle (2pt);
\draw[fill=black]  (15,1.732050807568879) circle (2pt);
\draw[fill=black] (14.60929333999749,-1.427493904129122) circle (2pt);
\end{tikzpicture}
\caption{$m_{G''}(-\infty,-2)=1$}
\end{subfigure}
\caption{Subdivisions of $G$}
\label{fig:1}
\end{figure}

We now consider the sequences $\{m_{G_t}(2,\infty)\}_{t\in \mathbb{N}}$, $\{m_{H_t}(2,\infty)\}_{t\in \mathbb{N}}$ and $\{m_{G_t}(-\infty,-2)\}_{t\in \mathbb{N}}$ and  $\{m_{H_t}(-\infty,-2)\}_{t\in \mathbb{N}}$. We first show that these sequences are (upper) bounded. 

\begin{proposition}\label{prop:Mbound} Let $Q=\{v\in G : \deg(v)\ge 3\}$. Then for any $t\in \mathbb{N}$, 
\[\max\{ m_{G_t}(2,\infty), m_{H_t}(2,\infty), m_{G_t}(-\infty,-2),  m_{H_t}(-\infty,-2) \} \leq |Q|.\] Moreover, this bound is tight.
\end{proposition}
\begin{proof}
Let $G_t-Q$ denote the subgraph of $G_t$ induced by the vertices $V(G_t)\backslash Q$. By Interlacing Theorem,
\[\lambda_j(G_t) \ge \lambda_j(G_t-Q) \ge \lambda_{|Q|+j}(G_t).\]
For $j=1$ we have
\[\lambda_{|Q|+1}(G_t) \le \lambda_1(G_t-Q) \le 2,\]
and for $j=|G_t|-|Q|$ we have
\[\lambda_{|G_t|-|Q|}(G_t)\ge \lambda_{|G_t|-|Q|}(G_t-Q)\ge -2.\]
This shows that $|Q|$ is an upper bound for $m_{G_t}(2,\infty)$ and $m_{G_t}(-\infty,-2)$. Similar argument works for $m_{H_t}(2,\infty)$ and $m_{H_t}(-\infty,-2)$.

Now if $S=E(G)$ and $G_t=G_t(S)$, then $m_{G_t}(2,\infty)=|Q|$ for all large $t$ by \autoref{thm:main3}. This example shows that the bound  for $m_{G_t}(2,\infty)$ is tight. The same example works for other cases.
\end{proof}

It is now easy to see that $\{m_{H_t}(2,\infty)\}$ and $\{m_{H_t}(-\infty,-2)\}$ are non-decreasing bounded integer sequences. By virtue of Propositions \ref{prop:sub_M} and \ref{prop:Mbound}, we also conclude that $\{m_{G_t}(2,\infty)\}$ is a non-decreasing bounded integer sequence. Hence, the following result.
\begin{theorem}
There exists $t_0\in \mathbb{N}$ such that $m_{G_t}(2,\infty), m_{H_t}(2,\infty)$ and $ m_{H_t}(-\infty,-2)$ are constant for all $t\ge t_0$.
\end{theorem} 

Unlike $m_{G_t}(2,\infty)$, which is a monotone sequence, the number of eigenvalues less than $-2$ may oscillate due to bipartiteness (see \autoref{fig:1}). In particular, it is not clear whether $m_{G_t}(-\infty, -2)$ stabilizes or not. 
\begin{conjecture}
There exists $t_0\in \mathbb{N}$ such that $m_{G_t}(-\infty, -2)$ is constant for all $t\ge t_0$.
\end{conjecture}

Since $H_t$ is an induced subgraph of $G_{2t+2}$ we have $\max_t m_{H_t}(2,\infty)\le \max_t m_{G_t}(2,\infty)$ and $\max_t m_{H_t}(-\infty, -2)\le \max_t m_{G_t}(-\infty, -2)$.
We believe that these inequalities are actually equalities and pose the following open problem.
\begin{conjecture} We have
 \[\max_t m_{H_t}(2,\infty)= \max_t m_{G_t}(2,\infty)\] and \[\max_t m_{H_t}(-\infty, -2)= \max_t m_{G_t}(-\infty, -2).\]   
\end{conjecture}

\section*{Acknowledgment}
Bojan Mohar is supported in part by the NSERC Discovery Grant R611450 (Canada),
and by the Research Project J1-2452 of ARRS (Slovenia).

\bibliographystyle{plainurl}
\bibliography{subdivision_references.bib}

\pagebreak

\section*{Appendix}
For the sake of completeness, we give a proof of \autoref{thm:main2} here. This proof mimics the proof of \autoref{thm:main1}. 

Let $n=|G_{2t+2}|.$ Then $|H_t|=n-|S|.$ For any fixed $k$, $\lambda_{n-k+1}(G_{2t+2})\le  \lambda_{n-|S|-k+1}(H_t)$ for all $t$, by Interlacing Theorem. In the following lemmas we show that $\lambda_{n-|S|-k+1}(H_t)$ cannot be much bigger than $\lambda_{n-k+1}(G_{2t+2})$ for large $t$.
\begin{lemma}\label{lem:d1} Let $k\in \mathbb{N}$ and $\eta\in \mathbb{R}^+$. Suppose $\{t_\ell\}_{\ell\in \mathbb{N}}$ is a sequence of natural numbers such that \[\lambda_{n-k+1}(G_{2t_\ell+2})\le -2-\eta\] for all $\ell$. Then for every $\varepsilon>0$ there exists $\ell_0$ such that for all $\ell\ge \ell_0$ we have
\[ \lambda_{n-|S|-k+1}(H_{t_\ell})\le -2-\eta +\varepsilon.\]
\end{lemma}
\begin{proof} Let $x^{(i)}=(x^{(i)}_v)_{v\in V(G_{2t+2})}$ be a normalized eigenvector for $\lambda_i(G_{2t+2})$. For each edge $uv$ in $S$, let $a,b$ be the middle vertices in the $t$-stretch of $uv$. Define $H_t'$ as the disjoint union of $H_t$ and $|S|$ many copies of $K_1$. Then for $n-k+1 \le i\le j\le n$ we have
\begin{align*}
     \langle x^{(i)}, A(H_t') x^{(j)} \rangle &= \langle x^{(i)}, A(G_{2t+2}) x^{(j)} \rangle - \sum_{uv\in S}\Bigl( x^{(i)}_ax^{(j)}_b + x^{(i)}_bx^{(j)}_a +  x^{(i)}_cx^{(j)}_b + x^{(i)}_bx^{(j)}_c \Bigr)\\
     &\le \delta_{ij} \lambda_i(G_{2t+2}) + \sum_{uv\in S}\Bigl(|x^{(i)}_a|+|x^{(i)}_b|+|x^{(i)}_c|\Bigr) \Bigl(|x^{(j)}_a| + |x^{(j)}_b|+|x^{(j)}_c| \Bigr).
\end{align*}
We may assume that $\varepsilon\le (k^2|S|)^{-1}$. By \autoref{lemma:decay}$(a)$, $|x^{(i)}_a|+|x^{(i)}_b|+|x^{(i)}_c|<\varepsilon$ for all large $t$ in our sequence, since $\lambda_i(G_{2t_\ell+2})<-2-\eta$, for all $i\ge n-k+1$. Thus we have 
\[\langle x^{(i)}, A(H_t') x^{(j)} \rangle\le \delta_{ij} \lambda_i(G_{2t+2}) + |S| \varepsilon^2.\]

Now let $W$ be the subspace spanned by $x^{(n-k+1)},\ldots,x^{(n)}$. Let 
\[z=\sum_{i=n-k+1}^n c_i x^{(i)},\quad \sum_{i=n-k+1}^n c_i^2=1.\]
Then,
\begin{align*}
    \langle z, A(H_t')z \rangle &= \sum_{i=n-k+1}^n \sum_{j=n-k+1}^n c_i c_j \langle x^{(i)}, A(H_t') x^{(j)} \rangle \\
    &\le \sum_{i=n-k+1}^n c_i^2 \lambda_i(G_{2t+2}) + \sum_{i=n-k+1}^n \sum_{j=n-k+1}^n |c_i||c_j ||S| \varepsilon^2\\
    &\le \lambda_{n-k+1}(G_{2t+2}) + \left(\sum_{i=n-k+1}^n |c_i|\right)^2 |S| \varepsilon^2\\
    &\le \lambda_{n-k+1}(G_{2t+2}) + k^2 |S| \varepsilon^2.
\end{align*}
Thus by Courant-Fischer Theorem,
\begin{align*}
    \lambda_{n-|S|-k+1}(H_t) &= \lambda_{n-k+1}(H_t')\\
    &\le \max_{\substack{z\in W\\||z||=1 }}\langle z, A(H_t') z\rangle\\
    &\le \lambda_{n-k+1}(G_{2t+2}) + k^2|S| \varepsilon^2\le -2-\eta +\varepsilon.
\end{align*}
This completes the proof.
\end{proof}

The following is immediate from the proof of \autoref{lem:d1}.
\begin{lemma}\label{lem:d2} Let $k\in \mathbb{N}$. Suppose there exist positive constants $t_0$ and $\eta$ such that 
\[ \lambda_{n-k+1}(G_{2t+2})\le -2-\eta\]
for all $t\geq t_0$. Then for any $\varepsilon > 0$ there exists $t_1$ such that whenever $t\geq t_1,$ 
\[ \lambda_{n-|S|-k+1}(H_t)\le \lambda_k(G_{2t+2})+\varepsilon. \]
\end{lemma}

We now prove \autoref{thm:main2}.
\begin{proof}[Proof of \autoref{thm:main2}]
We first show that 
\begin{equation}\label{eq:2evensequence}
    \lim_{t\to \infty}\lambda_{n-k+1}(G_{2t+2})=\lim_{t\to \infty}\lambda_{n-|S|-k+1}(H_t).
\end{equation}
Note that $\{H_t\}_{t\in \mathbb{N}}$ is a sequence of graphs with bounded maximum degree and that $H_t$ is an induced subgraph of $H_{t+1}$. Therefore, $\lambda_{n-|S|-k+1}(H_t)$ is a decreasing bounded sequence, hence Cauchy. Since the path $P_t$ is an induced subgraph of $H_t$, we have that $\lambda_{n-|S|-k+1}(H_t)\le \lambda_{t-k+1}(P_t)$ for all $t$ by Interlacing Theorem. Thus, 
\[\lim_{t\to \infty}\lambda_{n-|S|-k+1}(H_t)\leq \lim_{t\to \infty}\lambda_{t-k+1}(P_t)=-2.\]

First suppose that $\lim_{t\to \infty}\lambda_{n-|S|-k+1}(H_{t})=-2$. By \autoref{lem:d1} there exists a non-negative decreasing function $f(t)$ such that for all large $t$ we have $\lambda_{n-k+1}(G_{2t+2})\geq -2-f(t)$ and $f(t)\to 0$ as $t\to \infty$.  But $\lambda_{n-k+1}(G_{2t+2})\leq \lambda_{n-|S|-k+1}(H_t)$ for all $t$ and thus by Sandwich Theorem for limits we get $\lim_{t\to \infty}\lambda_{n-k+1}(G_{2t+2})=-2$.

Now suppose that $\lim_{t\to \infty}\lambda_{n-|S|-k+1}(H_{t})=-2-2\eta$ where $\eta$ is some positive constant. Then there exists $t_0$ such that for all $t\geq t_0$ we have $\lambda_{n-k+1}(G_{2t+2})\leq \lambda_{n-|S|-k+1}(H_{t})<-2-\eta$. Now for any $\varepsilon>0$, by \autoref{lem:d2}, we can find $t_1$ such that for $t\geq t_1,$
\[ |\lambda_{n-k+1}(G_{2t+2})-\lambda_{n-|S|-k+1}(H_{t})|<\varepsilon.\]
Thus $\lim_{t\to \infty}\lambda_{n-k+1}(G_{2t+2})=-2-2\eta$. This proves equality \eqref{eq:2evensequence}.

Now it remains to show that 
\begin{equation}\label{eq:2oddsequence}
\lim_{t\to \infty}\lambda_{|G_{2t+3}|-k+1}(G_{2t+3})=\lim_{t\to \infty}\lambda_{n-|S|-k+1}(H_t). 
\end{equation}
Let $G'$ be the graph obtained from $G$ by subdividing each edge $uv\in S$ once; let $w$ be the vertex of subdivision. Let $S'$ be the subset of edges of $G'$ that consists of $uw$ for each $uv\in S$. Let $H_t' = H_t(S')$ be the graph constructed from $G'$. By our discussion above we have 
\begin{equation}\label{eq:2odd1}
\lim_{t\to \infty}\lambda_{|G'_{2t+2}|-k+1}(G'_{2t+2})=\lim_{t\to \infty}\lambda_{|H_t'|-k+1}(H_t').    
\end{equation}
Now observe that $H_t$ is an induced subgraph of $H_t'$, and $H_t'$ is an induced subgraph of $H_{t+1}$. By Interlacing Theorem,
\begin{equation}\label{eq:2odd2}
\lambda_{|H_{t+1}|-k+1}(H_{t+1})\le \lambda_{|H_t'|-k+1}(H_t')\le \lambda_{|H_t|-k+1}(H_t).
\end{equation}
Equality \eqref{eq:2oddsequence} now follows from \eqref{eq:2odd1}, \eqref{eq:2odd2} and the fact that $G_{2t+2}'=G_{2t+3}$. This completes the proof.
\end{proof}
\end{document}